\documentclass[12pt]{amsart}

\usepackage{amsmath}
\usepackage{amssymb} 
\usepackage{stmaryrd}
\usepackage[all,cmtip]{xy}
\usepackage{xcolor}
\usepackage{ulem}

\newcommand{\Z}{{\mathbb Z}} 
\newcommand{\Q}{{\mathbb Q}} 
\newcommand{\C}{{\mathbb C}} 
 
\newcommand{\R}{{\mathbb R}}

\newcommand{\cF}{{\mathcal F}}

\newcommand{\cH}{{\mathcal H}}
\newcommand{\Mod}[1]{\ (\mathrm{mod}\ #1)}

\newtheorem{thm}{Theorem}

\newtheorem{prop}{Proposition}

\newtheorem{cor}{Corollary}

\theoremstyle{remark}

\newtheorem{rem}[]{Remark}

\newcommand\abs[1]{\left|#1\right|}

\newcommand{\itop}[2]{\genfrac {}{}{0pt}{3}{#1}{#2} }
\begin{document}  
\author{ P. Guerzhoy}
\author{Ka Lun Wong}
\address{ 
Department of Mathematics,
University of Hawaii, 
2565 McCarthy Mall, 
Honolulu, HI,  96822-2273 
}
\email{pavel@math.hawaii.edu}

\address{ 
Faculty of Math and Computing,
Brigham Young University - Hawaii, 
55-220 Kulanui Street, 
Laie, HI,  96762-1294
}
\email{kalun.wong@byuh.edu}

\title[]{Farkas' Identities with Quartic Characters}



\begin{abstract} 
Farkas in \cite{Farkas} introduced an arithmetic function $\delta$ and found an identity involving $\delta$ and a sum of divisor function $\sigma'$. The first-named author and Raji in \cite{Guerzhoy} 
discussed a natural generalization of the identity by introducing a quadratic character $\chi$ 
modulo a prime $p \equiv 3 \Mod 4$.  
In particular, it turns out that, besides the original case $p=3$ considered by Farkas, an exact analog (in a certain precise sense) of Farkas' identity
happens only for $p=7$. Recently, for  quadratic characters of small composite moduli, Williams in \cite{Williams} found a finite 
list of identities of similar flavor  using different methods.

Clearly, if $p \not \equiv 3 \Mod 4$, the character $\chi$ is either not quadratic or even. In this paper, we prove that, under certain conditions, no analogs
of Farkas' identity exist for even characters. Assuming $\chi$ to be odd quartic, we produce something surprisingly similar to the results from \cite{Guerzhoy}: 
 exact analogs of Farkas' identity happen exactly for $p=5$ and $13$.


\end{abstract}

\maketitle

\section{Introduction} \label{intro}
In 2004, Farkas \cite{Farkas} introduced an arithmetic function $\delta_F(n)$ which is defined as the difference between the number of positive divisors of $n$ that are congruent to 1 and $-1 \Mod 3$. He proved that for all positive integers $n$,
\begin{equation} \label{Farkas}
\delta_F (n)+3 \sum_{j=1}^{n-1} \delta_F (j) \delta_F (n-j) = \sigma'_3(n),
\end{equation}
where $\displaystyle \sigma'_3(n)=\sum_{\substack{d|n\\ 3 \nmid d}} d$.
This identity attracted interest, and was generalized in various directions by several authors. 
For a Dirichlet character $\chi$, define a function on positive integers by
\begin{equation} \label{delta_def}
\delta_{\chi}(n)=\sum_{0<d|n} \chi(d).
\end{equation}
Then  
\[
\delta_{F}(n)=\delta_\chi(n)
\]  
when $\chi$ is the quadratic character modulo $3$.
Recently K. Williams  in \cite{Williams} used  combinatorial arguments to prove 
$12$  identities similar to (\ref{Farkas}) while somehow more involved.
In these identities, function $\delta_\chi$ is associated with odd quadratic Dirichlet characters with small moduli such as
$3,4,8$, and $11$.

It is  convenient to define the quantities $\delta_{F}(0)=1/6$ and $\sigma'_3(0)=1/12$ so that 
Farkas' identity (\ref{Farkas}) becomes
\begin{equation} \label{Farkas_0}
\sum_{j=0}^n \delta_F (j) \delta_F (n-j) = \frac{1}{3}\sigma'_3(n)
\end{equation}
for all $n \geq 0$. For an odd quadratic character modulo a prime $p \equiv 3 \Mod 4$, it is proved  
in \cite{Guerzhoy} that an exact analog of Farkas' identity (\ref{Farkas_0}) holds if and only if $p=7$. 
However, if $p \equiv 1 \Mod 4$ then the quadratic character modulo $p$ is even, and methods 
and results of \cite{Guerzhoy} do not apply. That leaves a possibility that identities, possibly as simple and elegant as
the one originally written by Farkas, exist, and we indeed present some examples below.

Our modular forms interpretation of this kind of identities starts with an {\sl odd} Dirichlet character $\chi$. 
In order to keep our identities simple and elegant, we keep the assumption that the modulus of $\chi$ is a prime $p$.
(Consideration of composite moduli makes analogous identities more involved, therefore less transparent as one can see in \cite{Williams}.)
Since we now want to consider primes $p \not\equiv 3 \Mod 4$, we must pass from quadratic characters to characters
of higher order.

Let   $p \equiv 5 \Mod 8$ be  a prime, and consider a {\it quartic} (i.e. of exact order $4$) character modulo
$p$. There are exactly two of them, and we denote them by $\chi$ and $\overline{\chi}$ (here and throughout the bar denotes complex conjugation).  
Since $\chi$ is not real, there are two possibilities for an analog of the left hand side of (\ref{Farkas_0}).
We may thus write down two exact analogs of Farkas' identity (\ref{Farkas_0}) as follows:
\begin{equation} \label{id1}
\sum_{j=0}^n \delta_\chi (j) \delta_{\overline\chi} (n-j) = \alpha \sigma'_p(n)
\end{equation}
for some $\alpha \in \R$ and all $n \geq 0$, and
\begin{equation} \label{id2}
\sum_{j=0}^n \delta_\chi (j) \delta_\chi (n-j) = \alpha' \tilde{\sigma}_p(n) + \beta' \hat{\sigma}_p(n)
\end{equation}
for some $\alpha', \beta' \in \C$ and all $n \geq 0$.
The function $\sigma'_p$ in the right of the first identity, defined by
\[ 
\sigma'_p(n)=\sum_{\substack{0<d|n \\ p\nmid d}} d \hspace{3mm} \text{and} \hspace{3mm} \sigma'_p(0)=\frac{p-1}{24},
\]
generalizes the function $\sigma_3'$
in the original Farkas' identity directly while $\sigma$-functions in the second identity
are defined by
\[
\tilde{\sigma}_p(n)=\sum_{0<d|n} \left(\frac{p}{d}\right) d
\hspace{3mm} \text{and} \hspace{3mm} \tilde{\sigma}_p(0)=-\frac{1}{4}B_{2,\psi}
\]
and
\[
\hat{\sigma}_p(n)=\sum_{0<d|n} \left(\frac{p}{d}\right) n/d
\hspace{3mm} \text{and} \hspace{3mm} \hat{\sigma}_p(0)=0,
\]
where $B_{2,\psi}$ is the Bernoulli number associated with the Dirichlet character 
\[
\psi = \chi^2=\left( \frac{p}{\cdot} \right).
\]
While $\delta_\chi(n)$ for $n>0$ is defined in (\ref{delta_def}), we assume
\[
\delta_\chi(0) := -\frac{1}{2p} \sum_{a=1}^{p-1} \chi(a) a. 
\]

Our approach to these identities is close to that  proposed in \cite{Guerzhoy}. Specifically, we 
consider an Eisenstein series $G_{1,\chi}$ of weight $1$ such that $\delta_{\chi}(n)$ is its $n$-th Fourier coefficient.
Clearly, both $G_{1,\chi}^2$ and $G_{1,\chi}G_{1,\overline\chi}$  are 
weight $2$ modular forms. If the space of weight $2$ modular forms does not contain cusp forms, we obtain an identity 
when equate like powers of $q$ since we can explicitly write down the Fourier coefficients of weight two Eisenstein series.
Our assumptions about character $\chi$ keeps the dimension of weight $2$ Eisenstein series low. Specifically
(see e.g. \cite{Miyake} and \cite[Theorem 15.3.1]{Lang} for the Fourier expansions of Eisenstein series), in 
the absence of cusp forms,
\[
G_{1,\chi}^2 \in M_2(\Gamma_0(p), \chi^2) \hspace{3mm} \text{with} \hspace{3mm} \dim  M_2(\Gamma_0(p), \chi^2)  = 2,
\]
where $M_2(\Gamma_0(p), \chi^2)$ is generated by 
\[
\hat{G}_{2,p} = \sum_{n \geq 0} \hat{\sigma}_p(n) q^n \hspace{3mm} \text{and} \hspace{3mm}
\tilde{G}_{2,p} = \sum_{n \geq 0} \tilde{\sigma}_p(n) q^n,
\]
and 
\[
G_{1,\chi}G_{1,\overline{\chi}} \in M_2(\Gamma_0(p)) \hspace{3mm} \text{with} \hspace{3mm} \dim  M_2(\Gamma_0(p))  = 1,
\]
where $M_2(\Gamma_0(p)) $ is generated by
\[
G_2 = \sum_{n \geq 0} \sigma'_p(n) q^n.
\]
Here and throughout we assume $q=\exp(2 \pi i \tau)$ with $\Im (\tau)>0$. The identities (\ref{id1}) and (\ref{id2}) become
the identities for generating functions 
\begin{equation} \label{id1_m}
G_{1,\chi} G_{1,\overline{\chi}} = \alpha G_2
\end{equation}
and
\begin{equation} \label{id2_m}
G_{1,\chi}^2 = \alpha' \tilde{G}_{2,p} + \beta' \hat{G}_{2,p} 
\end{equation}
which one may consider simply as identities of formal power series in $q$. 

Note that the definitions of our arithmetic functions $\delta_\chi(n)$, $\sigma'_p(n)$, $\tilde{\sigma}_p(n)$, and
$\hat{\sigma}_p(n)$ for positive integers $n>0$ force our definitions above for their values at $n=0$. In other words, if the 
identities (\ref{id1}) and (\ref{id2}) hold true, these must be identities (\ref{id1_m}) and (\ref{id2_m}) between modular forms.
This is an exact analog of \cite[Proposition 1]{Guerzhoy}, and we skip the proof which is parallel to that given in \cite{Guerzhoy}.

It is easy to check that the absence of cusp forms, namely
\[
\dim S_2(\Gamma_0(p), \chi^2) = \dim  S_2(\Gamma_0(p))  = 0,
\]
happens 
(by a coincidence, simultaneously for both identities) 
if and only if $p=5$ or $13$. That implies (\ref{id1_m}) and (\ref{id2_m})  (equivalently, (\ref{id1}) and (\ref{id2}) ) and
constitutes the "if" part of our first result.

\begin{thm} \label{only}
Let   $p \equiv 5 \Mod 8$ be  a prime, and let $\chi$ be a quartic Dirichlet character modulo $p$.

The identities  \textup{(\ref{id1})} and \textup{(\ref{id2})} hold if and only if $p=5$ or $13$.

Specifically, if $p=5$ then for all $n \geq 0$
\[
\sum_{j=0}^n \delta_\chi (j) \delta_{\overline\chi} (n-j) = \frac{3}{5} \sigma'_p(n),
\]
\[
\sum_{j=0}^n \delta_\chi (j) \delta_\chi (n-j) = -\frac{4 + 3 \chi(2)}{10} \tilde{\sigma}_p(n) + \frac{2+\chi(2)}{2} \hat{\sigma}_p(n),
\]
and if $p=13$ then for all $n \geq 0$
\[
\sum_{j=0}^n \delta_\chi (j) \delta_{\overline\chi} (n-j) =  \sigma'_p(n),
\]
\[
\sum_{j=0}^n \delta_\chi (j) \delta_\chi (n-j) = -\frac{\chi(2)}{2} \tilde{\sigma}_p(n) + \frac{2+3\chi(2)}{2} \hat{\sigma}_p(n).
\]
\end{thm}

\begin{rem}
Although the statement of Theorem \ref{only} is parallel to the statement of \cite[Theorem 1]{Guerzhoy}, the method
of proof in loc. cit. which makes use of Minkowski estimate is not available in our setting. We developed an alternative, 
quite elementary method instead. It is possible to use a variation of this method in order to produce an alternative proof
of \cite[Theorem 1]{Guerzhoy}.
\end{rem}

We want to emphasize that, as one can see above, our approach allows us to find and prove the identities easily.
A more involved part of the proof of Theorem \ref{only} is to show that these exact analogs of Farkas' identities for quartic characters 
hold only for the primes $5$ and $13$. That is parallel to the principal result of \cite{Guerzhoy} where it was shown that 
such exact analogs of  Farkas' identities for quadratic characters hold only for the primes $3$ (which is the original Farkas' identity) and $7$.
We present a proof of the "only if" part of Theorem \ref{only} in Section 1 of the paper.

As in \cite{Guerzhoy}, Theorem \ref{only} entails a corollary pertaining to non-vanishing of certain special values of $L$-function 
associated to modular forms. Recall that for a cusp Hecke eigenform with Fourier expansion $f=\sum_{n>0} a(n)q^n$ and
a Dirichlet character $\xi$, associated $L$-functions are defined as the analytic continuation of the series
\[
L(s,f) = \sum_{n>0} a(n) n^{-s}, \hspace{5mm} L(s,f,\xi) = \sum_{n>0} \xi(n) a(n) n^{-s}.
\]

\begin{cor}\label{corLF}
For every prime $p > 13$ satisfying $p \equiv 5 \Mod 8$, and any quartic Dirichlet character $\chi$ modulo $p$, there exists a cusp Hecke eigenform $f \in S_2(p)$ such that 
\[
L(1,f)L(1,f,\chi) \not = 0
\]
and 
there exists a cusp Hecke eigenform $g \in S_2(p,\psi)$ with $\psi=\chi^2$ such that 
\[
L(1,g)L(1,g,\chi) \not = 0.
\]

\end{cor}
The proof of Corollary \ref{corLF} uses a variation of  Rankin's method as described in \cite{Shimura} and, since it is 
parallel to the proof of \cite[Theorem 4]{Guerzhoy}, we skip it.

As it is observed in \cite[Theorem 2]{Guerzhoy}, the identities, though never hold except for the two primes, are 
always not far from being true: the orders of magnitude of left and right hand sides are the same. That happens because
the obstruction for these identities to hold comes from cusp forms whose 
Fourier coefficients grow slower than those of Eisenstein series. 
In order to formulate a precise result, abbreviate the left-hand sides of our identities  (\ref{id1}) and (\ref{id2}) :
\[
\cF_{\chi}(n)= \sum_{j=0}^{n}\delta_{\chi}(j)\delta_{\overline{\chi}}(n-j),
\]
\[
\mathcal{H}_{\chi}(n)= \sum_{j=0}^{n}\delta_{\chi}(j)\delta_{\chi}(n-j) .
\]

The quantities $\alpha$, $\alpha'$, and $\beta'$ in (\ref{id1}) and (\ref{id2}) depend on the character $\chi$, and can be easily calculated (assuming that the identities hold true). 
We thus set
\[
\alpha = \frac{|\delta_{\chi}(0)|^2}{\sigma'_p(0)},
\]
and
\[
\alpha' = \frac{\delta_\chi(0)^2}{\tilde{\sigma}_p(0)} \hspace{3mm} \textup{and} \hspace{3mm}
\beta' = 2 \delta_\chi(0) -  \frac{\delta_\chi(0)^2}{\tilde{\sigma}_p(0)}.
\]

The limits in Theorem \ref{limit1_} are taken over $n$ not divisible by $p$, and within the range of a fixed value of the 
Kronecker symbol $\left(\frac{p}{n}\right)$ in part (b).

\begin{thm} \label{limit1_}
\textup{(a)}
For any prime $p \equiv 5 \Mod 8$ and a quartic Dirichlet character modulo $p$,
if $n$ is large and is not divisible by $p$, then the left and right sides of identity \textup{(\ref{id1})} are close.
Specifically, we have that
\begin{equation*}
\lim_{\itop{n \to \infty}{p \nmid n}} \frac{\cF_{\chi}(n)}{\sigma'_p(n)}=\alpha.
\end{equation*}

\textup{(b)} 
For any prime $p \equiv 5 \Mod 8$ and a quartic Dirichlet character modulo $p$,
if $n$ is large and is not divisible by $p$, then the left and right sides of identity \textup{(\ref{id1})} are closely related.
Specifically, we have that
\begin{equation*}
\lim_{\itop{n \to \infty}{p \nmid n}}\frac{\cH_{\chi}(n)}{\tilde{\sigma}_p(n)}= \alpha'+\left(\frac{p}{n} \right)\gamma.
\end{equation*}
with some $\gamma \in \C$.
\end{thm}

We prove Theorem \ref{limit1_} in Section 2.

Although Theorem \ref{limit1_} states
that identities are missed only narrowly, Theorem \ref{only} guarantees that 
exact analogs of Farkas' identities 
(\ref{id1}) and (\ref{id2}) 
for $p \equiv 5 \Mod 8$ 
hold only for $p=5$ and $13$. However, infinitely many similar, though less elegant, identities
hold for any such prime. Our next result is parallel to \cite[Theorem 3]{Guerzhoy}. The proof is also similar to that given in
\cite{Guerzhoy}. Specifically, one writes the identity which involves cusp forms and eliminates cusp forms with the help of action of 
Hecke operators.

\begin{thm} \label{many_id}

\textup{(a)}   
Let $t_p= \dim S_2(p)$.
There exist a set of complex numbers $A_i$ and two sets of positive integers $B_i$ and $C_i$ for $i=1,...,3^{t_p}$ (all three sets depend on specific $p \equiv 5 \Mod 8$) such that for any positive integer $n$
\begin{equation*}
\sum_{i=1}^{3^{t_p}} A_i\cF_{\chi}\left(\frac{n}{B_i}C_i\right)=\sigma'_p(n),
\end{equation*}
where we assume
\begin{equation*}
\cF_{\chi} \left(\frac{n}{B}C\right)=0
\end{equation*}
if $n$ is not divisible by $B$.

\textup{(b)}
Let $t_p= \dim S_2(p,\psi)$.
There exist a set of complex numbers $A_i$, two sets of positive integers $B_i$ and $C_i$ for $i=1,...,3^{t_p}$ (all three sets depend on specific $p \equiv 5 \Mod 8$), and two complex numbers, $\alpha''$ and $\beta''$, such that for any positive integer $n$
\begin{equation*}
\sum_{i=1}^{3^{t_p}} A_i\cH_{\chi} \left(\frac{n}{B_i}C_i\right)=\alpha'' \tilde{\sigma}_p(n)+\beta'' \hat{\sigma}_p(n),
\end{equation*}
where we assume
\begin{equation*}
\cH_{\chi} \left(\frac{n}{B}C\right)=0
\end{equation*}
if $n$ is not divisible by $B$.
\end{thm}

We omit details of the proof and provide an example (to Theorem \ref{many_id}(a)) instead.
Let $p=37$. Then $\dim S_2(37)=2$, and the space admits a basis out of two Hecke eigenforms with rational integer
coefficients. We make use of the fact that one of these cusp forms has zero Hecke eigenvalues at primes $p_1=2$ and $5$,
while another one has zero Hecke eigenvalues at $p_2=17$ and $19$. That allows us to produce the identities 
for $n > 0$:
\[
\cF_\chi(p_1p_2n)+p_1\cF_\chi(p_2n/p_1) + p_2\cF_\chi(p_1n/p_2) + p_1p_2\cF_\chi(n/(p_1p_2)) = 
\frac{1+p_1+p_2+p_1p_2}{3} \sigma'_{37}(n),
\]
where $\chi$ is a quartic character modulo $37$, and one may pick any $p_1 \in \{2,5\}$ and $p_2 \in \{17,19\}$.
As it is usual for arithmetic functions, we set $\cF_\chi(m)=0$ unless $m$ is a non-negative integer.

Our analysis of analogs for Farkas' identity relies heavily on their interpretation as identities for modular forms: it started  
with the interpretation of the values of the arithmetic function $\delta_\chi$ as Fourier coefficients of certain weight one Eisenstein series.
This interpretation is only possible if the Dirichlet character $\chi$ is odd. 
We believe that this interpretation explains all identities of this kind, and 
we present a partial theoretical result in this direction which we prove in Section 3 of the paper.
Note that while Theorem \ref{EvenChar_} has a fairly strong restriction on the modulus of {\it even} character $\chi$ under consideration, it has no
restrictions at all on its order.

\begin{thm}\label{EvenChar_}
Let $q \equiv 1 \Mod 4$ be a prime, and assume that $p=2q+1$ is a prime.
If $\chi$ be any even  Dirichlet character modulo $p$ then identity \textup{(\ref{id1})} cannot hold true for all $n \geq 1$.

\end{thm}

\section{Proof of Theorem \ref{only}} 

In this section, we prove the "only if" part of Theorem \ref{only}.

We start with identity (\ref{id1}), and we want to prove that it cannot be true for all $n \geq 1$ if $p>13$.
We actually shall show that if $p>13$, the identity (\ref{id1}) already cannot hold simultaneously for $n=1$ and $n=2$.

\begin{proof}[Proof of "only if" part for \textup{(\ref{id1})}]

For $n=1$ the identity reads

\begin{equation}\label{Coeq_}
\delta_{\chi}(1)\delta_{\overline{\chi}}(0)+\delta_{\chi}(0)\delta_{\overline{\chi}}(1)=\frac{|\delta_{\chi}(0)|^2}{\sigma'_p(0)} \sigma'_p(1).
\end{equation}
Note that $\delta_{\chi}(1)=\sigma'_p(1)=1$ and $\sigma'_p(0)=\frac{p-1}{24}$.
We abbreviate
\[
\delta_{\chi}(0)=-\frac{1}{2p} \sum_{a=1}^{p-1} \chi(a) a=\frac{L}{2} \in \Q(i).
\]
(In fact, by e.g.    \cite[Theorem 12.20]{Apostol}, $L=L(0,\chi)$, but we will not make use of any properties of Dirichlet $L$-function here.) 
Substitute all these quantities into (\ref{Coeq_}), and obtain that 
\begin{equation} \label{LEqu1_}
\frac{p-1}{6} \cdot \mathfrak{Re}(L)=|L|^2
\end{equation}

For $n=2$, the identity (\ref{id1}) reads
\[
\delta_{\chi}(2)\delta_{\overline{\chi}}(0) + \delta_{\chi}(1)\delta_{\overline{\chi}}(1)+\delta_{\chi}(0)\delta_{\overline{\chi}}(2)=\frac{|\delta_{\chi}(0)|^2}{\sigma'_p(0)} \sigma'_p(2).
\]
Note that the modulo $p$ character $\chi^2$ must be a non-trivial quadratic character modulo $p$, therefore $\chi^2$
must coincide with the Kronecker symbol.
In particular, 
\[
\chi(2)^2= \left(\frac{p}{2}\right)=-1
\] 
since $p \equiv 5 \Mod 8$, and therefore
\[
\chi(2)=\pm i
\]
Since obviously  $\sigma'_p(2)=3$, our identity for $n=2$
 transforms to

\begin{equation} \label{LEqu2_}
\frac{p-1}{18} \left[\mathfrak{Re}(L) \pm \mathfrak{Im}(L)+1\right]=|L|^2.
\end{equation}
It follows from (\ref{LEqu1_}) and (\ref{LEqu2_}) that 
\[
\mathfrak{Im}(L)= \pm [2\mathfrak{Re}(L)- 1].
\]
Denote $R=\mathfrak{Re}(L)$, and note that $R \in \Q$ since $L \in \Q(i)$.
We thus have
\[
L=R \pm (2R-1)i,
\]
and equation (\ref{LEqu1_}) now becomes
\begin{equation*}
\frac{p-1}{6}\cdot R= R^2+(2R-1)^2
\end{equation*}
and transforms to
\begin{equation}
30R^2-(p+23)R+6=0.
\end{equation}
Since a rational $R$ satisfies this quadratic equation, its discriminant $(p+23)^2-4(30)(6)=(p+23)^2-720$ being an integer, must be a  perfect square, that is
\[
(p+23)^2-720=x^2,
\]
with a positive integer $x$.
Equivalently, 
\[
(p+23+x)(p+23-x)=720.
\]
We now look at all  factorizations of $720$ into a product of two positive integers, and find all possible values of
$p$ and $x$. We find that $p=5$ and $p=13$ are the only possibilities.

\end{proof}

We now pass to the identity (\ref{id2}).  The proof is a bit more technical though it is based on a similar combination of ideas.
Specifically, we want to show that, if $p>13$, the identity cannot hold simultaneously for $n=1$,$2$, and $3$.

We will need some information about the rational number $B_{2,\psi}$ involved in the definition
of $\tilde{\sigma}_p(0) = - B_{2,\psi}/4$ above.

\begin{prop} \label{Bernoulli}
For a prime $p \equiv 5 \Mod 8$ with $p>13$ and the quadratic Dirichlet character $\psi$ modulo $p$, the generalized Bernoulli number $B_{2,\psi}$ is a rational number bigger than $4$. When $p=5$,  $B_{2,\psi}=4/5$ and when $p=13$,  $B_{2,\psi}=4$.
\end{prop}

\begin{proof}
Recall that $\psi=\chi^2$ is the non-trivial quadratic character modulo $p$ (the Kronecker symbol).
For the generalized Bernoulli number $B_{2,\psi}$ (one has that $L(-1,\psi)= - B_{2,\psi}/2$) Cohen proved in \cite{Cohen} that
\[
B_{2,\psi} = -2H(2,p),
\]
where the quantity $H(2,p)$ can be calculated as
\[
H(2,p)=-\frac{1}{5} \sum_{s} \sigma \left(\frac{p-s^2}{4}\right),
\]
where the sum runs over all integers $s$ such that $p>s^2$ and $\sigma$ stands for the usual divisor sum function.
In particular, one easily finds that
\[
B_{2,\psi}= \begin{cases}
4/5 & \text{if $p=5$,} \\
4 & \text{if $p=13$.}
\end{cases}
\]

If $p>13$ then $p \geq 29$, and we have that
\[
B_{2,\psi}  = \frac{2}{5} \sum_{s} \sigma \left(\frac{p-s^2}{4}\right) \geq   
\frac{4}{5}\left( \sigma \left(\frac{p-1}{4} \right) \right) \geq \frac{32}{5} >4
\]
because $(p-1)/4 \geq 7$ and 
\[
\sigma \left(\frac{p-1}{4} \right) \geq 1 + 7 = 8.
\]
\end{proof}

We are now ready to finish the proof of Theorem \ref{only}.

\begin{proof}[Proof of "only if" part for \textup{(\ref{id2})}]

We start by using (\ref{id2}) for $n=0$ and $1$ to find the quantities $\alpha'$ and $\beta'$ and plug them into the 
identity (\ref{id2}) for $n=2$:

\[
2\delta_{\chi}(0)\delta_{\chi}(2)+\delta_{\chi}(1)\delta_{\chi}(1)=\frac{\delta_{\chi}(0)^2}{\tilde{\sigma}_p(0)} \tilde{\sigma}_p(2)+\frac{2\delta_{\chi}(0)\delta_{\chi}(1)-\frac{\delta_{\chi}(0)^2}{\tilde{\sigma}_p(0)}\tilde{\sigma}_p(1)}{\hat{\sigma}_p(1)}\hat{\sigma}_p(2).
\]

As previously, we use the obvious values 
$\delta_{\chi}(1)=\tilde{\sigma}_p(1)=\hat{\sigma}_p(1)=1$, $\delta_{\chi}(2)=1 + \chi(2)$, $\hat{\sigma}_p(2)=1$ 
and $\tilde{\sigma}_p(2)=-1$ along with $\tilde{\sigma}_p(0)=-\frac{1}{4}B_{2,\psi}$ and simplify the identity to

\begin{equation}\label{QuadEq_}
\frac{\delta_{\chi}(0)^2}{-\frac{1}{4} B_{2,\psi}} = -\chi(2)\delta_{\chi}(0)-\frac{1}{2}.
\end{equation}

A similar calculation simplifies the identity (\ref{id2}) for $n=3$ to
\[
2(1+\chi(3))\delta_{\chi}(0)+2(1 +\chi(2))=\frac{\delta_{\chi}(0)^2}{-\frac{1}{4}B_{2,\psi}} (1+3\psi(3))+\left(2\delta_{\chi}(0)-\frac{\delta_{\chi}(0)^2}{-\frac{1}{4}B_{2,\psi}}\right)(3+\psi(3)),
\]
and we combine it with 
(\ref{QuadEq_}) to obtain

\begin{align*}
2(1+\chi(3))\delta_{\chi}(0)+2(1 +\chi(2) )&=\left(-\chi(2) \delta_{\chi}(0)-\frac{1}{2}\right)(1+3\psi(3)) & \\
&+\left(2\delta_{\chi}(0) +\chi(2) \delta_{\chi}(0) + \frac{1}{2}\right)(3+\psi(3)).  \\
\end{align*}

Since $\chi$ is a quartic character, $\chi(3) \in \{\pm 1,\pm i\}$. The above equation allows us to find the quantity 
$\delta_\chi(0)$ which corresponds to every one of these four values. With this quantity, we make use of 
equation (\ref{QuadEq_}) to find the corresponding values of $B_{2,\psi}$. These turn out to be either $4$, or $4/5$, or
a complex number which is not rational. 
Proposition \ref{Bernoulli} now implies that the only possibilities left are $p=5$ and $p=13$ as required.



\end{proof}

\section{Proof of Theorem \ref{limit1_}} 

In this section, we prove Theorem \ref{limit1_}. We start with part (a) which is simpler. This proof is  parallel to
the proof of Theorem 2 in \cite{Guerzhoy}. However, we present it here because that proof in \cite{Guerzhoy} has a typo.

\begin{proof}[Proof of Theorem \textup{ \ref{limit1_}(a) }]
We keep the notations from Introduction. 
Since $G_{1,\chi}G_{1,\overline{\chi}} \in M_2(p)$, we may always write
\[
G_{1,\chi}G_{1,\overline{\chi}}=\alpha G_2 +f
\]
with $\alpha=\frac{|\delta_{\chi}(0)|^2}{\sigma'_p(0)}$, and a cusp form $f\in S_2(p)$. 
Let 
\[
f=\sum_{n > 0} a(n)q^n.
 \]
Equating the coefficients of $q^n$, we obtain for $n \geq 0$
\begin{equation} \label{Fcoe_}
\cF_{\chi}(n)=\alpha \sigma'_p(n) + a(n).
\end{equation}
We can see that $\sigma'_p(n)>n$ if $(p,n)=1$ from its definition. Now we need an upper bound for the Fourier coefficients of cusp forms. The Ramanujan-Petersson conjecture proved by Deligne implies that
\[
|a(n)|<M\sqrt{n}
\]
with some constant $M$
Now, if we divide (\ref{Fcoe_}) by $\sigma'_p(n)$ and take the limit, the theorem follows.
\end{proof}

\begin{rem}
The special case  of Ramanujan-Petersson conjecture for weight $2$ which we make use of here 
was proved by Shimura \cite{Shimura2}.  Analogous proof of Theorem 2 in \cite{Guerzhoy} refers to
Hecke ("trivial") estimate for the coefficients of cusp forms, and this estimate actually does not suffice for the proof.
\end{rem}

The proof of  Theorem \ref{limit1_}(b) is a bit more subtle both because we will have to deal with two arithmetic functions, 
$\tilde{\sigma}_p$ and $\hat{\sigma}_p$, instead of just $\sigma'_p$, and because the functions themselves are slightly 
more complicated. We start with a proposition which relates these two functions.

\begin{prop} \label{sigmas_}
For a prime $p$ and a positive integer $n$ not divisible by $p$,
\[
\hat{\sigma}_p(n)=\left(\frac{p}{n}\right) \tilde{\sigma}_p(n).
\]
\end{prop}
\begin{proof}
Since $\hat{\sigma}_p(n)$, $\tilde{\sigma}_p(n)$ and  $\left(\frac{p}{n}\right)$ are all multiplicative functions, we only need to consider the case when $n$ is a prime power. Let $n=l^r$ for some prime $l \neq p$ and some positive integer $r$. 
By definition, we have
\begin{align*}
\hat{\sigma}_p(n)=\hat{\sigma}_p(l^r) & = \sum_{t=0}^r \left(\frac{p}{l^t}\right) l^{r-t} =  \sum_{t=0}^r \left(\frac{p}{l}\right)^t l^{r-t} = 
 \sum_{u=0}^r \left(\frac{p}{l}\right)^{r-u} l^{u}  =  \left(\frac{p}{l}\right)^{r} \sum_{u=0}^r \left(\frac{p}{l}\right)^{-u} l^{u} \\
& =  \left(\frac{p}{l^r}\right) \sum_{u=0}^r \left(\frac{p}{l}\right)^{u} l^{u} =  \left(\frac{p}{n}\right) \tilde{\sigma}_p(n) 
\end{align*}
because for a positive integer $u$
\[
\left(\frac{p}{l^u}\right) = \left(\frac{p}{l}\right)^u \hspace{3mm} \text{and} \hspace{3mm} 
\left(\frac{p}{l}\right)^{-u} = \left(\frac{p}{l}\right)^u.
\]
\end{proof}

We thus have that $|\hat{\sigma}_p(n)|=|\tilde{\sigma}_p(n)|$, for $p \nmid n$. and we will use a lower bound for this quantity
given in the next proposition.

\begin{prop} \label{estimate}
For a positive integer $n$ not divisible by $p$,
\[
\abs{ \tilde{\sigma}_p(n)} \geq n (1/2)^{\omega(n)},
\]
where $\omega(n)$ is the number of distinct prime factors of $n$.
\end{prop}

\begin{proof}
Since arithmetic function $\tilde{\sigma}_p(n)$ is multiplicative, it suffices to prove that, for a prime $l$ and a positive integer $r$,
\[
\abs{ \tilde{\sigma}_p(l^r)} \geq \frac{1}{2}l^r.
\]
Indeed, 

\[
\abs{ \tilde{\sigma}_p(l^r)} = \abs{ \sum_{u=0}^r \left(\frac{p}{l}\right)^{u} l^{u} } =
\abs{ \frac { 1- \left(\frac{p}{l}\right)^{u+1} l^{u+1} }{ 1- \left(\frac{p}{l}\right) l }} \geq \frac{l^{r+1}-1}{l+1}  \geq \frac{1}{2} l^r,
\]
where the latter inequality is equivalent to $l \geq 1 + 2l^{-r}$.
\end{proof}

We are now in a position to prove Theorem \ref{limit1_}(b).

\begin{proof}[Proof of Theorem \textup{ \ref{limit1_}(b) }]
Since $G^2_{1,\chi} \in M_2(p,\psi)$ with 
\[
\psi(n)=\chi^2(n) = \left(\frac{p}{n}\right),
\]
we can write 
\[
G^2_{1,\chi}=\alpha' \tilde{G}_{2,p} +\gamma \hat{G}_{2,p}+f
\]
with some $\gamma \in \C$,  
a cusp form $f \in S_2(p,\psi)$, and
\[
\alpha' = \frac{\delta_\chi(0)^2}{\tilde{\sigma}_p(0)}
\]
because neither $f$ nor $\hat{G}_{2,\psi}$ has a constant term in the Fourier expansion.

We equate the coefficients of $q^n$, divide both parts of the equation by $\tilde{\sigma}_p(n)$, and use Proposition \ref{sigmas_}
to obtain 

\[
\frac{\mathcal{H}_{\chi}(n)}{\tilde{\sigma}_p(n)}=\alpha' +  \left(\frac{p}{n}\right)\gamma + \frac{a(n)}{\tilde{\sigma}_p(n)}.
\]

In order to finish the proof, it now suffices to show that
\begin{equation} \label{lim}
\lim_{n \to \infty} \frac{a(n)}{\tilde{\sigma}_p(n)}=0.
\end{equation}

As in part (a) above, we still have that  
$|a(n)|\leq M\sqrt{n}$ for some constant $M$,
and we make use of Proposition \ref{estimate} to get, for $n$ big enough,
\begin{equation*}
\left|\frac{a(n)}{\tilde{\sigma}_p(n)}\right| \leq \frac{M\sqrt{n}}{\left(\frac{1}{2}\right)^{\omega(n)}n}= \frac{M2^{\omega(n)}}{\sqrt{n}}.\\
\end{equation*}
We now make use of a bound proved in  \cite{Robin}
 \[
 \omega(n) <  13841 \frac{\ln n}{\ln (\ln n)}
 \] 
 to obtain that
\begin{align*}
\left|\frac{a(n)}{\tilde{\sigma}_p(n)}\right| < \frac{M \cdot 2^{13841 \frac{\ln n}{\ln (\ln n)}}}{\sqrt{n}}= \frac{M \cdot n^{\frac{13841\ln 2}{\ln \ln n}}}{\sqrt{n}} = M \cdot n^{\frac{13841 \ln 2}{\ln \ln n} - \frac{1}{2}} \rightarrow 0
\end{align*}
as $n \rightarrow \infty$. That implies (\ref{lim}) and concludes the proof.
\end{proof}

\section{Proof of Theorem \ref{EvenChar_}} 

In this Section, we prove Theorem \ref{EvenChar_}.
We want to prove that, for an even character $\chi$, the identity (\ref{id1}) does not hold simultaneously already for $n=1$ and $n=p$.
Specifically, assuming (\ref{id1}) is true for $n=1$, we find that 
$\alpha = \delta_{\chi} (0)+\delta_{\overline{\chi}} (0)$, and
the identity for $n=p$ simplifies to
\[
\sum_{j=1}^{p-1} \delta_\chi(j) \delta_{\overline{\chi}}(p-j) = 0.
\]

Theorem \ref{EvenChar_} follows immediately from the following proposition.

\begin{prop} \label{even}
Let $p=2q+1$ and $q \equiv 1 \Mod 4$ be a prime. Let $\psi$ be an even Dirichlet character modulo $p$.
Then 
\begin{equation} \label{contra}
\sum_{j=1}^{p-1} \delta_\psi(j) \delta_{\overline{\psi}}(p-j) \neq 0.
\end{equation}
\end{prop}

The rest of this section is devoted to the proof of Proposition \ref{even}.

Our next proposition implies, in particular, that (\ref{id1}) is always true for $n=p$ if the character $\chi$ is odd.

\begin{prop}\label{ZeroSum_}

\textup{(a)}
Let $p$ be a prime and $\chi$ be a Dirichlet character modulo $p$. For any $1\leq j \leq p-1$, the expression 
$\delta_{\chi}(j)\delta_{\overline{\chi}}(p-j)$ is purely imaginary if $\chi$ is odd, and real if $\chi$ is even. 

\textup{(b)}
If the character $\chi$ is odd, then 
\[
\sum_{j=1}^{p-1} \delta_\chi(j) \delta_{\overline{\chi}}(p-j) = 0.
\]
\end{prop}

\begin{proof}
The expression in question can be rewritten as a sum
\[
\delta_{\chi}(j)\delta_{\overline{\chi}}(p-j) = \sum_{d|j} \sum_{d'|p-j} \chi(d) \overline{\chi}(d') = 
\frac{1}{2}\sum_{d|j, \hspace{1mm} d'|(p-j)} \left( \chi(d) \overline{\chi}(d') + \chi(j/d) \overline{\chi}((p-j)/d') \right).
\]
Making use of $\chi(j)\overline{\chi}(-j) = \chi(-1)$ we transform every summand
\begin{align*}
  \chi(d) \overline{\chi}(d') + \chi(j/d) \overline{\chi}((p-j)/d')  & = \chi(d) \overline{\chi}(d') + \chi(j)\overline{\chi}(d)\overline{\chi}(-j)\chi(d') \\
       &=  \chi(d) \overline{\chi}(d') + \chi(-1) \overline{ \chi(d) \overline{\chi}(d') },
\end{align*}
and  assertion (a) becomes clear term-by term. Assertion (b) follows from that since the sum is real (because it is equal to its conjugate).
\end{proof}

 


We now begin to exploit some specifics of our assumptions about the prime $p$.

\begin{prop} \label{2prim}
Let $p=2q+1$ and $q \equiv 1 \Mod 4$ be a prime. Then $2$ is a primitive root modulo $p$
(i.e. a generator of $(\Z/p\Z)^*$).
\end{prop}
\begin{proof}
The subgroup of squares has index $2$ in $(\Z/p\Z)^*$, thus there are exactly $(p-1)/2=q$ non-squares modulo $p$.
At the same time, there are exactly
\[
\varphi(\varphi(p)) = \varphi(p-1) = \varphi(p-1) = \varphi(2q) = \varphi(2) \varphi(q) = q-1
\]
primitive roots modulo $p$. Since no square can be a primitive root, all but one non-squares must be primitive roots.
The non-square which is not a primitive root is $-1$ (since $p \equiv 3 \Mod 4$, the residue $-1$ is indeed a non-square
modulo $p$).
By quadratic reciprocity, $2$ is a quadratic non-residue 
modulo $p$, and since it is different from $-1$, it must be a primitive root modulo $p$.
\end{proof}

Since now on we assume that
our prime  $p=2q+1$, where $q \equiv 1 \Mod 4$ is a prime, and  let $\zeta = \exp( 2 \pi i/(p-1))$.
The group of Dirichlet characters modulo $p$ is cyclic of order $p-1$ generated by the (odd) character $\chi$ defined by
\[
\chi(2) = \zeta.
\]

We now construct certain polynomials associated with a Dirichlet character $\xi$ modulo $p$.
For an arbitrary  Dirichlet character $\xi$ modulo $p$, define integers 
\[
0 \leq t(\xi,d) \leq p-2
\]
by 
\[
\xi(d) = \zeta^{t(\xi,d) }.
\]
Note that, for any positive integers $k$ and $d$, 
\begin{equation} \label{cong_t}
t(\xi^k,d) \equiv kt(\xi,d) \mod (p-1)
\end{equation}
Let, for a positive integer $j\leq p-1$,
\[
h_{\xi,j}(x) :=\sum_{d | j} x^{t(\xi,d)} 
\]
be a polynomial (in $x$) of degree at most $p-2$, and let
\begin{align*}
  f_{\xi}(x) &:= \sum_{j=1}^{p-1} h_{\xi,j} (x)h_{\bar{\xi},p-j}(x)  \\
       &= b_{2p-4}x^{2p-4}+b_{2p-3}x^{2p-3}+\ldots+b_px^p+b_{p-1}x^{p-1}+b_{p-2}x^{p-2}+\ldots+b_1x+b_0
\end{align*}
be a polynomial of degree at most $2p-4$.

The polynomials just introduced allow us to reformulate Proposition \ref{even}. Clearly,
\[
h_{\xi,j}(\zeta) = \delta_\xi(j).
\]
For a character $\xi=\chi^k$ with a positive integer $k$, we obtain making use of (\ref{cong_t})
\[
h_{\chi^k,j}(\zeta)=\delta_{\chi^k}(j)=\sum_{d |j} \chi^k (d) = \sum_{d |j} (\chi (d))^k = \sum_{d |j} (\zeta ^{t(\chi,d)})^k 
= \sum_{d |j} (\zeta ^k)^{t(\chi,d)} = h_{\chi,j}(\zeta^k).
\]
It follows  that 
\[
\sum_{j=1}^{p-1} \delta_{\chi^k}(j) \delta_{\overline{\chi^k}}(p-j) 
=f_{\chi^k}(\zeta) = f_{\chi}(\zeta^k).
\]
Since every Dirichlet character $\xi$ modulo $p$ can be written as $\xi = \chi^k$, where the parity of $\xi$ coincides with the parity of $k$, we deduce from Proposition \ref{ZeroSum_} that $f_{\chi}(\zeta^k)=0$ if $k$ is odd, and our target Proposition \ref{even} can be reformulated as follows.

\begin{prop} \label{even'}
Let $\zeta=\exp(2\pi i/(p-1))$. 
Let $\chi$ be the character modulo $p$ defined by its value on the primitive root $\chi(2)=\zeta$. 
Let $f_{\chi}(x)$ be the polynomial associated with $\chi$ as above.
Then, for a positive integer $k$, we have that  $f_{\chi}(\zeta^k)=0$ (i.e. the quantity $\zeta^k$ is a root of the polynomial $f_{\chi}(x)$) if and only if $k$ is odd.
\end{prop}

Our ultimate goal now is to prove Proposition \ref{even'}.  We need some specific information about the coefficients of $f_{\chi}(x) $
given in the following proposition.
\begin{prop} \label{coeff}
The polynomial $f_{\chi}(x) $ has positive integer coefficients. Furthermore,
\[
b_0=p-1, \hspace{5mm} b_1=\frac{p-1}{2}+1, \hspace{5mm} b_p=b_{p-1}=0.
\]
\end{prop}

We postpone the proof of Proposition \ref{coeff} and prove Proposition \ref{even'} (therefore Proposition \ref{even}, therefore
Theorem \ref{EvenChar_}) now.

\begin{proof}[Proof of Proposition \ref{even'} ]
Let 
\[
g(x)=b_{p-2}x^{p-2}+(b_{2p-4}+b_{p-3})x^{p-3}+\cdots+(b_p+b_1)x+(b_{p-1}+b_0).
\]
Then the polynomial $(x^{p-1}-1)$ divides the polynomial $f_{\chi}(x) - g(x)$, and 
therefore $\zeta^k$ is a root of $f_{\chi}(x)$ if and only if it is a root of $g(x)$.
In the factorization
\[
x^{p-1}-1= x^{2q}-1 = (x^q-1)(x^q+1),
\]
$\zeta^k$ with odd $k$ are exactly the roots of the second factor (while those with even $k$ are exactly the roots of the first factor). 
Since  $f_{\chi}(\zeta^k)=0$ if $k$ is odd, we have that
\[
g(x)=(x^q+1)P(x)
\]
with a polynomial $P(x)$ with integer coefficients of degree at most $q-1$. If an even power of $\zeta$ was a root of $g(x)$, it would be a root of $P(x)$ thus $P(x)$ would  be divisible by the cyclotomic polynomial $\Phi_q(x)=x^{q-1}+x^{q-2}+ \cdots + x^2 +x+1$
because $q$ is a prime. Since, however, $\deg P(x) \leq \deg \Phi(x)=q-1$, the two polynomials would differ by a constant factor.
If this was the case then, for $P(x)=a_{q-1}x^{q-1} + \ldots + a_1x+a_0$, we would have $a_1=a_0$, which translates immediately to
\[
b_p+b_1=b_{p-1}+b_0
\]
for the coefficients of  $f_{\chi}(x) $ in contradiction with Proposition \ref{coeff}.

\end{proof}

We are left only to prove Proposition \ref{coeff}.

\begin{proof}[Proof of Proposition \ref{coeff}]

We start with recording some values of the character $\chi$.
Since 
\[
\chi(2)=\zeta,
\]
\[
\zeta\chi(q)=\chi(2q) = \chi(p-1) = -1 = \zeta^q, \hspace{3mm} \text{implies} \hspace{3mm} \chi(q)=\zeta^{q-1};
\]
\[
\chi(2)\chi(q+1) = \chi(2q+2) = \chi(p+1) = 1 \hspace{3mm} \text{implies} \hspace{3mm} \chi(q+1)=\zeta^{p-2}.
\]
We thus have the following values of $t(\chi,d)$:
\begin{center}
\begin{tabular}{l | r}
$d$ & $t(\chi,d)$ \\
\hline
$1$ & $0$ \\
$2$ & $1$ \\
$q$ & $q-1$ \\
$q+1$ & $p-2$ \\
\end{tabular}
\end{center}

By definition,
\[
f_{\chi}(x) = \sum_{j=1}^{p-1} h_{\chi,j} (x)h_{\overline{\chi},p-j}(x) 
\]
and since both $h_{\chi,j} (x)$ and $h_{\overline{\chi},p-j}(x)$ have constant terms $1$, the constant term of $f_{\chi}(x) $ is
\[
b_0=p-1.
\]
In order to calculate $b_1$, note that  $h_{\chi,j} (x)$  has an $x$-term (with coefficient $1$) every time when $2 | j$ while
$h_{\overline{\chi},p-j}(x)$ has an  $x$-term (with coefficient $1$) only when $j=q$ (otherwise $(q+1) \nmid p-j = 2q+1-j$).
We thus have all together
\[
b_1=\frac{p-1}{2}+1.
\]
Note that for every $d$ such that $ 1 < d \leq p-1$, there exists exactly one solution $u$ such that 
 $ 1 < u \leq p-1$ and $t(\chi,d)+t(\overline{\chi},u) = p-1$, and that is $u=d$ since $\chi(d) \overline{\chi}(d) = 1$.
 It follows that 
 \[
 b_{p-1}=0
 \]
 because no $d>1$ can divide simultaneously $j$ and $p-j$.
 
We now claim that $b_p=0$. Otherwise we would have that
\[
t(\chi,d)+t(\overline{\chi},y) = p.
\]
That implies
\[
t(\chi,y) = t(\chi,d)-1,
\]
and therefore
\[
d \equiv 2y \Mod p.
\]
Since both $2 \leq d,y \leq p-2$,
either $d=2y$ or $d=2y-p$ with $y > (p+1)/2$.

However, $d=2y$ is not possible since $2y | j$ and $y | (p-j)$ at a time would imply $y |p$.

We are thus left with the only option that $y | (p-j)$ and $(2y-p) | j$ while $y > (p+1)/2$ which implies $y=p-j$.
Then $2y-p=p-2j$ and we can write
\[
(p-2j)t = j
\]
with some positive integer $t$.
We find that
\[
j=\frac{pt}{1+2t},
\]
and conclude that $p=2t+1$ because $(t,2t+1)=1$.
Thus
\[
(p-2j)\frac{p-1}{2} = j,
\]
and that implies
\[
j=\frac{p-1}{2} \hspace{3mm} \text{therefore} \hspace{3mm} y=p-j=\frac{p+1}{2}
\] 
in contradiction with $y>(p+1)/2$ above.
\end{proof}

\end{document}